\documentclass[draft]{amsart}

\usepackage{amssymb}
\usepackage{url}
\usepackage{amscd}
\usepackage{amsmath,amssymb,amsfonts}
\usepackage{mathrsfs}
\usepackage{pifont}
\input xy
\xyoption{all}
\usepackage{amsaddr}
\title{On Invariants
of $\text{C}^*$-algebras with the ideal property}
\author{Kun Wang}
\address{Department of Mathematics, Texas A\&M University\\
College Station, TX, US, 77843
}
\email{kwang@math.tamu.edu}


\theoremstyle{plain}
\newtheorem{theorem}{Theorem}[section]

\newtheorem{corollary}[theorem]{Corollary}
\newtheorem{lemma}[theorem]{Lemma}
\newtheorem{proposition}[theorem]{Proposition}
\newtheorem{definition}[theorem]{Definition}
\theoremstyle{remark}
\newtheorem{remark}[theorem]{Remark}
\newtheorem{example}[theorem]{Example}

\begin{document}


\begin{abstract}
In this paper, we study the relation between the extended Elliott invariant and the Stevens invariant of $\text{C}^*$-algebras.
We show that in general the Stevens invariant can be derived from the extended Elliott invariant in a functorial manner.
We also show that these two invariants are isomorphic for $\text{C}^*$-algebras satisfying the ideal property. 
A $\text{C}^*$-algebra is said to have the ideal property if each of its closed two-sided ideals is generated by projections inside the ideal.
Both simple, unital  $\text{C}^*$-algebras and real rank zero  $\text{C}^*$-algebras have the ideal property.
As a consequence, many classes of non-simple  $\text{C}^*$-algebras can be classified by their extended Elliott invariants, 
which is a generalization of Elliott's conjecture. 
\end{abstract}

\maketitle


\newcommand\sfrac[2]{{#1/#2}}
\newcommand\cont{\operatorname{cont}}
\newcommand\diff{\operatorname{diff}}

\section{Introduction}
George Elliott initiated the classification program of nuclear $\text{C}^*$-algebras 
since his  classification of approximately finite-dimensional (AF) algebras via
their scaled, ordered $\text{K}_0$-groups (\cite{E1}).
Successful classification results have been obtained for
 $\textrm{AH}$ algebras (the inductive
limits of matrix algebras over metric spaces) with slow dimension growth for cases of  real rank zero (see \cite{EGLP 1}, \cite{EGLP 2}, \cite{EG}, \cite{EG 2}, \cite{Da}, \cite{Go 1}, \cite{Go 2}, \cite{DL}, \cite{Ei}, \cite{DG})  and simple $\textrm{AH}$
algebras (see \cite{E}, \cite{Ell 1}, \cite{Li 1}, \cite{Go 3}, \cite{EGL})  by using the so called Elliott invariant, which consists of
the ordered $\text{K}_0$-group, the $\text{K}_1$-group, the simplex of tracial state space and the natural pairing beween the tracial state space and the $\text{K}_0$-group.

A $\text{C}^*$-algebra is said to have the ideal property if each of its closed two-sided ideals is generated (as a closed two-sided ideal) by projections inside the ideal.
It is obvious  that both simple, unital  $\text{C}^*$-algebras and real rank zero  $\text{C}^*$-algebras have the ideal property.
There are many other examples of $\text{C}^*$-algebras arising from dynamical systems which have the ideal property (see \cite{GJLP}, \cite{Pa2}, \cite{Pa3}, \cite{PP1}, \cite{PP2}, \cite{PR}, etc.).
In 1995,
K. Stevens classified all unital approximately divisible
$\textrm{AI}$ algebras with the ideal property (\cite{Ste}). 
Conel Pasnicu studied
$\text{C}^*$-algebras with the ideal property and obtained a
characterization theorem for $\textrm{AH}$ algebras with the ideal property which are quite useful for classification theory (see \cite{Pa1}).
In 2011, K. Ji and C. Jiang improved  Stevens' result
by dropping the conditions unital and approximate divisible (see \cite{JJ}). 
Subsequently, Jiang and the present author completely classified all inductive limits of splitting interval algebras (\text{ASI}) with the ideal property (\cite{JW}).
AI algebra is a special case of $\text{ASI}$ algebra.
The invariant we used to classify \textrm{ASI} algebras in our paper was first proposed by Stevens. 
 We  call it  Stevens invariant.
Stevens invariant of a $\text{C}^*$-algebra  $A$ consists of the $\text{K}_0$-group of $A$, the $\text{K}_1$-group of $A$ and the  tracial state spaces of all hereditary $\text{C}^*$-subalgebras of the form $\overline{eAe}$ with certain compatibility conditions, where $e$ is any projection in $A$. 
Stevens invariant is also used to classify  $\text{AH}$-algebras with the ideal property (see \cite{GJL}).

We know that for simple $\text{C}^*$-algebras, traces are assumed to be bounded in the unital cases, 
and lower semicontinuous and densely defined in the non-unital case. 
But these two kinds of traces will not suffice for the classification of non-simple $\text{C}^*$-algebras. That is, for non-simple $\text{C}^*$-algebras, in many cases,
 all finite traces or densely defined lower semi-continuous traces are identically zero on a proper ideal. 
 Therefore, neither finite traces nor densely defined traces can  give information of the ideals. 
In this paper, we propose to include the extended valued traces (the value could be infinity) in the traditional Elliott invariant---called {\bf Extended Elliott Invariant.}
Another sign for considering the extended valued traces is that all lower semicontinuous traces on a $\text{C}^*$-algebra
constitute a non-cancellative cone that in particular determines the lattice of closed two-sided ideals, an important invariant in its own right. 

 It is natural to ask what is the connection between  the extended Elliott invariant and the Stevens invariant?
 Does the extended Elliott invariant still work for classifying $\text{C}^*$-algebras with the ideal property?
In this paper, we partially answered these two questions.
The following theorems are our main results in this paper:
\begin{theorem}\label{mainthm1}
Let $A$ be a $\text{C}^*$-algebra with the ideal property. 
Then the Stevens invariant of $A$ is equivalent to the extended Elliott invariant of $A$.
\end{theorem}

\begin{theorem}\label{mainthm2}
 Let $A,~B$ be two $\text{C}^*$-algebras with the ideal property. If $A$ and $B$ have isomorphic extended Elliott Invariant, then $A$ and $B$ have isomorphic  Stevens Invariant---and vise versa.
\end{theorem}

The paper is organized as follows. In Section 2, we recall some definitions and lemmas. 
In Section 3, we define two categories $\mathcal{S}$ and $\mathcal{E}$ corresponding to  Stevens invariant and extended Elliott invariant  respectively.
We show that there are canonical non-trivial maps between the object set of $\mathcal{E}$ and the object set of $\mathcal{S}$.  
Moreover, the Stevens invariant of a $\text{C}^*$-algebra can always be derived from its extended Elliott invariant.
The converse is  true when the 
$\text{C}^*$-algebra has the ideal property.
In Section 4, we extend the maps defined in Section 3 to be functors between 
two sub-categories of $\mathcal{S}$ and $\mathcal{E}$ and prove  Theorem 1.2.
Finally, we show that there is a class of $\text{C}^*$-algebras without the ideal property whose extended Elliott invariants  cannot be derived from their Stevens invariants. \\

{\bf Acknowledgement.}
The result of this paper is part of my thesis. 
I benefit a lot from discussions with my advisor, Professor Guihua Gong in writing this paper. 
I would like to express my gratitude to him, for his support, patience, and encouragement throughout my graduate studies.

\section{Preliminaries}
For convenience of the reader, we recall some definitions and lemmas (see \cite{Pedersen} for more details).

\begin{definition}\label{weight} Let $A$ be a $\text{C}^*$-algebra. A weight on $A$ is a function $\phi:A_+\rightarrow [0,+\infty]$ such that\\
  (i) $\phi(\alpha x)=\alpha\phi(x)$, if $x\in A_+$ and $\alpha\in \mathbb{R}_+$;\\
 (ii) $\phi(x+y)=\phi(x)+\phi(y)$, if $x$ and $y$ belong to $A_+$.\\
Moreover, $\phi$ is lower semi-continuous if for each $\alpha\in\mathbb{R}_+$ the set $$\{x\in A_+|~\phi(x)\leq\alpha\}$$ is closed.
\end{definition}

\begin{definition}\label{trace} Let $A$ be a $\text{C}^*$-algebra. A trace on $A$ is a weight $\phi$ such that $\phi(u^*xu)=\phi(x)$ for all $x\in A_+$ and all unitary $u\in \widetilde{A}$, where $\widetilde{A}$ is the unitization of $A$.
\end{definition}

\begin{remark}
In this paper, we denote by $\text{T}(A)$ the collection of all lower semicontinuous traces on $A$. This set is a non-cancellative cone endowed with operations of pointwise addition and pointwise scalar multiplication by strictly positive real numbers (see \cite{ERS} for details). Let $\text{T}_{\text{F}}(A)$ denote the set of all finite traces on $A$.
\end{remark}

The following two propositions are properties of traces quoted from \cite{Pedersen}.

\noindent Let $A$ be a $\text{C}^*$-algebra and $\phi$ be a trace on $A$. Let $A_+^{\phi}$ be a subset of $A_+$ defined by $$A_+^{\phi}:=\{x\in A_+|~\phi(x)<\infty\}.$$
\begin{proposition}\label{fin}(see 5.1.2 of \cite{Pedersen}) For each trace $\phi$ on a $C^*$-algebra $A$ the linear span $A^{\phi}$ of $A_+^{\phi}$ is a not necessary closed ideal of $A$ with $(A^{\phi})_+=A_+^{\phi}$, and there is a unique extension of $\phi$ to a positive linear functional on $A^{\phi}$. Moreover, the set $A_2^{\phi}=\{x\in A|~x^*x\in A_+^{\phi}\}$ is an ideal of $A$ such that $y^*x\in A^{\phi}$ for any $x,y\in A_2^{\phi}$. 

\end{proposition}

\begin{proposition}\label{comm}(see 5.2.2 of \cite{Pedersen}) If $\phi$ is a trace on a $C^*$-algebra $A$ then $\phi(yx)=\phi(xy)$ for each $x$ in $A^{\phi}$  and $y$ in $\widetilde{A}$. If moreover $\phi$ is lower semi-continuous then $\phi(x^*x)=\phi(xx^*)$ for all $x$ in $A$ and $\phi(xy)=\phi(yx)$ for all $x$ and $y$ in $A_2^{\phi}$.
\end{proposition}

Next we want to discuss how to extend traces.
\begin{definition} We define an equivalence relation in $A_+$ by setting $x\approx y$ if there is a finite set $\{z_n\}$ in $A$ such that $x=\sum z_n^*z_n$ and $y=\sum z_nz_n^*$.  And we use the notation $y\preccurlyeq x$ to mean $y\approx x_1,~x_1\leq x.$
\end{definition}
\begin{theorem}\label{extend}(see 5.2.7 of \cite{Pedersen}) Let $B$ be a hereditary $C^*$-subalgebra of a $C^*$-algebra $A$, and let $\phi$ be a lower semi-continuous weight on $B$. For each $x$ in $A_+$ define $$\widetilde{\phi}(x)=\sup\{\phi(y)|~y\in B_+,y\preccurlyeq x\}.$$ Then $\widetilde{\phi}$ is a lower semi-continuous trace on $A$ and $\widetilde{\phi}|_{B_+}$ is the smallest trace dominating $\phi$.
\end{theorem}

The following two definitions are some usual notations.
\begin{definition}
Let $A$ be a $\text{C}^*$-algebra. 
Let $\mathcal{P}(A)$ be the set of all projections in $A$.
Let $\text{K}_0(A)$ be the $\text{K}_0$-group of $A$
and $K_0(A)^+\subseteq K_0(A)$ be the semigroup of $K_0(A)$
generated by $[p]\in \text{K}_0(A)$, where $p\in \mathcal{P}_{\infty}(A)$.
Define  $$\Sigma A =\{[p]\in K_0(A)^+: ~p\mbox{ is a projection in }A\}.$$
\end{definition}
Then $( \text{K}_0(A),  \text{K}_0(A)^+,\Sigma A)$ is a scaled ordered group.

\begin{definition}
Let $X$ be any convex set. 

\begin{enumerate}
\item Let $\textrm{Aff}(X)^+$ be the collection of all affine maps from $X$ to $[0,\infty]$.
\item Let $\textrm{Aff}_b(X)$ be the collection of all affine maps from $X$ to $\mathbb{R}$.
\item Let $\textrm{Aff}_b(X)^+$ be the subset of $\textrm{Aff}_b(X)$ consisting of all nonnegative affine functions.
\end{enumerate}

Any affine map $\xi:X\rightarrow Y$ induces a linear map $\xi^*:\textrm{Aff}(Y)\rightarrow \textrm{Aff}(X)$ by 
$$\xi^*(f)(\tau)=f(\xi(\tau)),$$
for all $f\in \textrm{Aff}(Y)$ and $\tau\in X.$

\end{definition}

\section{Two invariants and their relevant categories }
In this section, we construct two categories $\mathcal{E}$ and $\mathcal{S}$ in which 
the extended Elliott invariant and Stevens invariant sit, respectively. 
We show that there is a canonical non-trivial map from the object set of   $\mathcal{E}$ to the object set of  $\mathcal{S}$, which induces a map from 
the extended Elliott invariant of a stably finite $\text{C}^*$-algebra to its Stevens invariant.
We also construct a canonical map from the object set of   $\mathcal{S}$ to the object set of  $\mathcal{E}$, which induces a map from 
the Stevens  invariant  to the extended Elliott invariant for a  stably finite $\text{C}^*$-algebra with the ideal property.

Let $\mathcal{E}$ denote the category whose objects are four-tuples 
$$((G_0,G_0^+),\Sigma G, G_1, X),$$
where $(G_0,G_0^+)$ is a partially ordered abelian group; $G_1$ is a countable abelian group; $\Sigma G$ is a subset of $G^+_0$;
$X$ is a cone  
 closed under addition and positive scalar multiplication such that there exists a positive linear map
$s^G$ from $G_0^+$ to  $\textrm{Aff}(X)^+$. 
And $X$ is also a complete lattice cone when endowed with the order structure induced by 
its addition operation (i.e., $\tau_1\leq \tau_2$ if there exists $\tau_3\in X$ such that $\tau_1+\tau_3=\tau_2$).

A morphism $$\Theta:((G_0,G_0^+),\Sigma G, G_1, X)\rightarrow ((H_0,H_0^+),\Sigma H, H_1, Y)$$
in $\mathcal{E}$ is a three-tuple $$\Theta=(\theta_0,\theta_1,\zeta),$$
where $\theta_0:(G_0,G_0^+,\Sigma G)\rightarrow (H_0,H_0^+,\Sigma H)$ is an order-preserving homomorphism satisfying
$\theta_0(\Sigma G)\subseteq \Sigma H$; $\theta_1:G_1\rightarrow H_1 $ is any homomorphism and $\zeta:Y\rightarrow X$ 
is a continuous affine map that makes the diagram below commutative:
$$\xymatrix@C=1cm{G_0^+\ar[d]^{s^G}\ar[r]^{\theta_0}&H_0^+\ar[d]^{s^H}\\
\textrm{Aff}(X)\ar[r]^{\zeta^*}&\text{Aff}(Y)}$$

\begin{definition}
For a  $\text{C}^*$-algebra $A$, the extended Elliott invariant of $A$ is
$$((\text{K}_0(A),\text{K}_0(A)^+),\Sigma A, \text{K}_1(A),\text{T}(A)),$$
with the natural pairing between $\text{K}_0(A)^+$ and $\text{T}(A)$, i.e., let $$s^A:\text{K}_0(A)^+\rightarrow \text{Aff}(\text{T}(A))$$
be defined by evaluating a given trace at a $\text{K}_0$-class.
\end{definition}
Obviously, if $A$ is a stably finite $\text{C}^*$-algebra, then the extended Elliott invariant of $A$ is an object in the category $\mathcal{E}$
(see \cite{ERS} for more details).
Given a class  of  stably finite $\text{C}^*$-algebras, say $\mathcal{A}$, let $\mathcal{E}_{\mathcal{A}}$ denote 
the subcategory of $\mathcal{E}$ whose objects can be realised as the extended Elliott invariant of a member of $\mathcal{A}.$

\begin{definition}
A preordered cone $(F,\leq)$ is said to have the Riesz property if for $f,g,h\in F$
with $f\leq g+h$ there are always $\widehat{g},\widehat{h}\in F$ with $\widehat{g}\leq g$, $\widehat{h}\leq h$ such that $f=\widehat{g}+\widehat{h}.$
\end{definition}

\begin{proposition} (see Theorem 2.6.8 in \cite{FL})\label{wolf}
Let $(F,\leq)$ be a lattice cone. Then its positive dual cone $F^*_+$ has the Riesz property and it is a complete lattice cone. 
\end{proposition}

Let $\mathcal{S}$ denote the category whose objects are four-tuples 
$$((G_0,G_0^+),\Sigma G, G_1, \{\Delta_p^G\}_{p\in G_0^+}),$$
where $(G_0,G_0^+)$ is a partially ordered abelian group; $G_1$ is a countable abelian group; $\Sigma G$ is a subset of $G^+_0$;
for each $p\in G_0^+$, there is a  positive cone  $\Delta_p^G$ with a base of a simplex, and a positive linear map 
$$s_p^G:G_0^p\longrightarrow \textrm{Aff}_b(\Delta_p^G),$$ 
where $G_0^p$ is the subgroup of $G_0$ generated by the set 
$$\{e\in G_0:0\leq e\leq np\mbox{ for some } n\in\mathbb{Z}\}.$$
For any $p'\in G_0^+$ with $p'\leq p$,
 there is an affine map $\lambda_{p,p'}^G:\Delta_p^G\rightarrow \Delta_{p'}^G$ 
 satisfying the following conditions:

(1) If $p''\leq p'\leq p$, then 
 $\lambda_{p,p''}^G=\lambda^G_{p',p''}\circ\lambda_{p,p'}^G.$

(2) $s^G_p(e)(\tau)=s^G_{p'}(e)(\lambda_{p,p'}(\tau))$ for all $\tau\in\Delta_p^G$, $e\in G_0^{p'}.$

(3) The map $(\lambda_{p,p'}^G)^*:\text{Aff}_b(\Delta_{p'}^G)\rightarrow \text{Aff}_b(\Delta_{p}^G)$ induced by $\lambda_{p,p'}^G$ is hereditary, i.e.,
 if $f\in \text{Aff}_b(\Delta_{p'}^G)$ and $g\in \text{Aff}_b(\Delta_{p}^G)$ satisfying $(\lambda_{p,p'}^G)^*(f)\geq g$, then
there exists $h\in \text{Aff}_b(\Delta_{p'}^G)$ such that $$g=(\lambda_{p,p'}^G)^*(h).$$

(4) For each $f\in \text{Aff}_b(\Delta_{p+q})$, there exist $f_1\in \text{Aff}_b(\Delta_p), f_2\in\text{Aff}_b(\Delta_q)$ such that 
$$f=\lambda_{p+q,p}^*(f_1)+\lambda_{p+q,q}^*(f_2)$$
where $p,q \in G_0^+$. 

\begin{remark}
By the definition of $\Delta_p^G$ in the Stevens invariant, we know that $\Delta_p^G$ 
is a lattice cone for each $p$.
\end{remark}

A morphism $$\Theta:((G_0,G_0^+),\Sigma G, G_1, \{\Delta_p^G\}_{p\in G_0^+})\rightarrow ((H_0,H_0^+),\Sigma H, H_1,  \{\Delta_e^H\}_{e\in H_0^+})$$
in $\mathcal{S}$ is a three-tuple $$\Theta=(\theta_0,\theta_1,\{\xi^p\}_{p\in G_0^+}),$$
where $\theta_0:(G_0,G_0^+,\Sigma G)\rightarrow (H_0,H_0^+,\Sigma H)$ is an order-preserving homomorphism satisfying
$\theta_0(\Sigma G)\subseteq \Sigma H;$ $\theta_1:G_1\rightarrow H_1 $ is any homomorphism;
for each $p\in G_0^+$, there is a continuous affine map
$\xi^p:\Delta_{\theta_0(p)}^H\rightarrow \Delta_{p}^G$ 
 that makes the diagram below commutative:
$$\xymatrix@C=1cm{\Delta_{q}^H\ar[d]^{\lambda_{q,q'}^H}\ar[r]^{\xi^p}&\Delta_{p}^G\ar[d]^{\lambda_{p,p'}^G}\\
\Delta_{q'}^H\ar[r]^{\xi^{p'}}&\Delta_{p'}^G}$$
where $q=\theta_0(p)$, $q'=\theta_0(p')$, $p,p'\in G_0^+$ satisfying $p'\leq p.$

\begin{definition}
For a 	$\text{C}^*$-algebra $A$, the Stevens invariant is 
$$((\text{K}_0(A),\text{K}_0(A)^+),\Sigma A, \text{K}_1(A),\{\text{T}_{\text{F}}(\overline{pAp})\}_{p\in \text{K}_0(A)^+}),$$
with a natural pairing $s^A_p$ between $\text{K}_0(A)^p$ and $\text{T}_{\text{F}}(\overline{pAp})$ given by evaluating a given trace at a $\text{K}_0$-class and 
$\lambda_{p,q}^A:\text{T}_{\text{F}}(\overline{pAp}) \rightarrow \text{T}_{\text{F}}(\overline{qAq})$  defined by restriction. That is,
$$\lambda_{p,q}^A(\tau)=\tau|_{\overline{qAq}}.$$
\end{definition}

It is easy to see that the Stevens invariant of $A$ is an object in the category $\mathcal{S}$ when $A$ is stably finite.
Given a class $\mathcal{A}$ of stably finite $\text{C}^*$-algebras, let $\mathcal{S}_{\mathcal{A}}$ denote 
the subcategory of $\mathcal{S}$ whose objects can be realised as the Stevens invariant of a member of $\mathcal{A}.$


\begin{lemma}(see Lemma 10.4 in \cite{Ph})\label{l1} If $P$ is a lattice and $P_1$ is a hereditary subcone of $P$, then
$P_1$ is a lattice.

\end{lemma}

\begin{lemma}\label{l2}
Let $((G_0,G_0^+),\Sigma G, G_1, X)$ be an object in $\mathcal{E}$.
For any $p\in G_0^+$, 
let $$\Delta'_p=\{\tau\in X:0\leq s^G(p)(\tau)<\infty\},$$
$$\Delta_p^G=\Delta'_p/\sim,$$
where $\tau_1\sim\tau_2$ if and only if $f(\tau_1)=f(\tau_2) $ for all $f\in \text{Aff}(X)$ satisfying 
$f$ is bounded on $\Delta_p'.$
Then $\Delta_p^G$ is a positive cone with a base of a simplex.
\end{lemma}

\begin{proof}
For $\tau\in\Delta'_p$, let $[\tau]$ represent the equivalent class of $\tau$ in $\Delta_p^G.$ 
It's easy to see that  for all $\alpha\in\mathbb{R}_+$ and $\tau\in \Delta'_p$,  $$\alpha[\tau]=[\alpha\tau]\in \Delta_p^G.$$
If $[\tau_1],$ $[\tau_2]\in \Delta_p^G$ and $t\in[0,1],$ then $$t[\tau_1]+(1-t)[\tau_2]=[t\tau_1+(1-t)\tau_2]\in \Delta_p^G.$$  
Therefore, $\Delta_p^G$ is a positive cone. 
Similarly, $\Delta'_p$ is a positive subcone of $X$. 

Claim: $\Delta_p'$ is a hereditary subcone of $X$. 
If $\phi_1\in\Delta_p'$, $\phi_2\in X$ satisfying $\phi_2\leq\phi_1$, then by the definition of the order on $X$
there exists $\phi_3\in X$ such that $$\phi_1=\phi_2+\phi_3.$$
Thus,
\begin{align*}
s^G(p)(\frac{1}{2}\phi_1)&=s^G(p)(\frac{1}{2}\phi_2+\frac{1}{2}\phi_3)\\
&=\frac{1}{2}s^G(p)(\phi_2)+\frac{1}{2}s^G(p)(\phi_3).\\
\end{align*}
Since $s^G(p)(\phi_1)<\infty$, we have $s^G(p)(\phi_2)<\infty$. 
Thus, $\phi_2\in \Delta_p'$ and the claim is true.
By Lemma \ref{l1}, $\Delta_p'$ is a lattice.

Define $\|\tau\|=s^G(p)(\tau)$ for all $\tau\in \Delta_p'.$ 
It is easy to see that $\|\cdot \|$ is a norm on $\Delta_p'$.
Therefore, $\Delta_p'$ can be embedded into a norm space. Let
$$T_p'=\{\tau\in\Delta_p':s^G(p)(\tau)=1\}.$$
Then $T_p'$ is a convex base of the cone $\Delta_p'$ with $\Delta_p'$ a lattice. 
Therefore, $T_p'$ is a simplex (see the second paragraph of page 52 in \cite{Ph}).

Let $T_p^G=T_p'/\sim$, which is a simplex base of $\Delta_p^G.$

\end{proof}

\begin{theorem}\label{DefG}
There is a natural nontrivial transformation $\mathcal{G}$ that maps the objects in $\mathcal{E}$ to the objects in $\mathcal{S}$.
\end{theorem}

\begin{proof}
Let $((G_0,G_0^+),\Sigma G, G_1, X)$ be any object in $\mathcal{E}$.
For any $p\in G_0^+$, define
$\Delta_p'$ and $ \Delta_p^G$ as in the Lemma \ref{l2}.
Then $\Delta_p^G$ is a positive cone with a base of a simplex.

Let $$G_0^{p+}=G_0^p\cap G_0^+$$
and $s_p^G:G_0^{p+}\rightarrow \textrm{Aff}_b(\Delta_p^G)^+$ be defined by 
$$s_p^G(q)([\tau])=s^G(q)(\tau)$$
for $q\in G_0^{p+}$ and $\tau\in \Delta_p'.$
Since $s_p^G(q)(\tau)<\infty$ for all $q\in  G_0^{p+},\tau\in\Delta_p'$, we can extend $s_p^G$ to be a map (still denote it by $s_p^G$) from $G_0^p$ to  $\textrm{Aff}_b(\Delta_p^G)$ by
$$s_p^G(q_1-q_2)([\tau])=s_p^G(q_1)(\tau)-s_p^G(q_2)(\tau)$$
for $q_1,q_2\in G_0^{p+}$ and $\tau\in \Delta_p'.$
It is easy to see that $s_p^G$ is a well-defined linear map. 

For $q\in G_0^+$ with $q\leq p$, let $\lambda_{p,q}^G:\Delta_p^G\rightarrow\Delta_q^G$  be the map induced by the inclusion map from $\Delta_p'$ to $\Delta_q'$.
It is easy to see that $\lambda_{p,q}^G$ is well-defined and moreover, \\
(1) If $p''\leq p'\leq p$, then 
 $\lambda_{p,p''}^G=\lambda^G_{p',p''}\circ\lambda_{p,p'}^G.$\\
(2) $s^G_p(e)([\tau])=s^G_{p'}(e)(\lambda_{p,p'}([\tau]))$ for all $\tau\in\Delta_p'$, $e\in G_0^{p'}.$\\
(3) If $f\in \text{Aff}_b(\Delta_{q}^G)$ and $g\in \text{Aff}_b(\Delta_{p}^G)$ satisfying $(\lambda_{p,q}^G)^*(f)\geq g$,
let $T_q'$ be a simplex base of $\Delta_q'$ and let $E_q'$ be the set of extreme points of $T_q'$.
Define
 $h:E_q'\rightarrow\mathbb{R} $ by the following
$$
h( \tau) = \left\{
  \begin{array}{l l}
     g(\tau) & \quad \text{if $\tau\in \Delta_p'\cap E_q'$}\\
     f(\tau) & \quad \text{otherwise}
   \end{array} \right.
$$ 
Then $h$ can be extended to an affine map from $\Delta_q'$ to $\mathbb{R}$ (still denoted by $h$). 
Since any element in $\Delta_p'$ can not be written as a linear combination of elements in $\Delta_q'\setminus\Delta_p'$, 
$h$ induces  a map from  $\Delta_q^G$ to $\mathbb{R}$ satisfying    
  $$g=\lambda_{p,q}^*(h)\mbox{ and } h\leq f.$$ 
 Thus, the map $(\lambda_{p,p'}^G)^*:\text{Aff}(\Delta_{q}^G)\rightarrow \text{Aff}(\Delta_{p}^G)$ is hereditary.\\
(4) Let $p,q$ be any two elements in $G^+_0$.
 Since $\Delta_{p+q}'$ is a lattice cone, by Proposition \ref{wolf}, $\text{Aff}_b(\Delta_{p+q}')$ 
has the Riesz property. 
For each $f\in \text{Aff}_b(\Delta_{p+q})$, there exists a constant $M$ such that $$f\leq M(s^G(p)+s^G(q)).$$
By the Riesz property, there exist 
 $f_1\in \text{Aff}_b(\Delta_p), f_2\in\text{Aff}_b(\Delta_q)$ such that 
$$f=\lambda_{p+q,p}^*(f_1)+\lambda_{p+q,q}^*(f_2).$$
Therefore,  $((G_0,G_0^+),\Sigma G, G_1, \{\Delta_p^G\}_{p\in G_0^+})$ is an object in $\mathcal{S}.$
Finally let $$\mathcal{G}:\{\mbox{Objects in }\mathcal{E}\}\rightarrow\{\mbox{Objects in }\mathcal{S}\}$$ be the map defined by
sending
$$((G_0,G_0^+),\Sigma G, G_1, X)\mapsto((G_0,G_0^+),\Sigma G, G_1, \{\Delta_p^G\}_{p\in G_0^+}),$$
which completes the proof.

\end{proof}

\begin{corollary}\label{welG}
Let $\mathcal{A}$ be the class of stably finite $\text{C}^*$-algebras. Then
$$\mathcal{G}(\mathcal{E}(A))\cong\mathcal{S}(A)$$
for all $A\in\mathcal{A}$.
\end{corollary}

\begin{proof}
Let $A\in\mathcal{A}.$ 
We know that
$$\mathcal{E}(A)=((\text{K}_0(A),\text{K}_0(A)^+),\Sigma A, \text{K}_1(A),\text{T}(A)).$$
Let $\mathcal{G}(\mathcal{E}(A))=((\text{K}_0(A),\text{K}_0(A)^+),\Sigma A,\text{K}_1(A), \{\Delta_p^G\}_{p\in \text{K}_0(A)^+}).$
By the definition of $\mathcal{G}$, for each $p\in \text{K}_0(A)^+,$ 
$$\Delta_p'=\{\tau\in \text{T}(A):0\leq\tau(p)<\infty\},$$
$$\Delta_p^A=\Delta_p'/\sim,$$
where $\tau_1\sim\tau_2$ if and only if $f(\tau_1)=f(\tau_2)$ for all $f\in \text{Aff}(\text{T}(A))$ satisfying $f$ is bounded on $\Delta_p'$.
It is enough to show that  $$\Delta_p^{A}\cong \text{T}_{\text{F}}(\overline{pAp})$$ for each $p\in\text{K}_0(A)^+.$

For any $\tau\in \Delta_p'$, $\tau|_{\overline{pAp}}$ is a finite trace on $\overline{pAp}$. Define
 $$\gamma:\Delta_p^{A}\rightarrow \text{T}_{\text{F}}(\overline{pAp})$$
by $$\gamma([\tau])=\tau|_{\overline{pAp}}.$$
Thus $\gamma$ is a positive affine map and it is easy to see that 
  $\gamma$ is well-defined.
  
  Notice that any element in $\text{AffT}(\overline{pAp})$ can be realized by self-adjoint element in $\overline{pAp}$.
  Thus, if $[\tau_1]\neq [\tau_2]$ in $\Delta_p^A$, then there exists an element $f\in \text{AffT}(\overline{pAp})$ and a self-adjoint element $a\in\overline{pAp}$ such that $$\tau_1(a)=f(\tau_1)\neq f(\tau_2)=\tau_2(a).$$
  Therefore, $\gamma([\tau_1])\neq \gamma([\tau_2])$. $\gamma$ is injective.

What's left is to show that $\gamma$ is a surjection. Let $\phi$ be any finite trace on $\overline{pAp}$. By Theorem \ref{extend}, there is a lower semi-continuous trace $\widetilde{\phi}$ 
on $A$ such that $\widetilde{\phi}|_{\overline{pAp}}=\phi.$ Thus, $$\gamma(\widetilde{\phi})=\phi.$$
Therefore, $\gamma$ is a one-to-one and onto affine map. Thus, we have 
$$\mathcal{G}(\mathcal{E}(A))\cong \mathcal{S}(A).$$

\end{proof}

\begin{theorem}\label{DefF}
There is a natural nontrivial transformation $\mathcal{F}$ that maps the objects in $\mathcal{S}$ to the objects in $\mathcal{E}$.
\end{theorem}

\begin{proof}
Let $((G_0,G_0^+),\Sigma G, G_1, \{\Delta_p^G\}_{p\in G_0^+})$ be any object in $\mathcal{S}$.
We say that $\Lambda$ is an ideal of $G_0^+$ if  it is a hereditary sub-semigroup of $G_0^+$.
Let $X$ be the collection of all sets of the form
$\{\tau_p\}_{p\in \Lambda}$ satisfying\\
(1) $\Lambda$ is an ideal of $G_0^+$;\\
(2) $\tau_p\in\Delta_p$ for each $p\in\Lambda$ and 
 $\lambda_{p,q}(\tau_p)=\tau_q\mbox{ whenever } q<p.$\\
 The multiplication in $X$ is  the usual multiplication and the addition operation in $X$ is defined as follows:
 $$\{\tau_p\}_{p\in\Lambda_1}+ \{\phi_p\}_{p\in \Lambda_2}=\{\tau_p+\phi_p\}_{p\in \Lambda_1\cap\Lambda_2}.$$
We define an order relation $\preceq$ on $X$ by the following:
$$\{\tau_p\}_{p\in\Lambda_1}\preceq \{\phi_p\}_{p\in \Lambda_2}\mbox{ if and only if }
\Lambda_1\supseteq \Lambda_2\mbox{ and }\tau_p\leq\phi_p \mbox{ for }p\in \Lambda_2.$$
Then $X$ naturally has a lattice structure induced by the lattice structures of $\Delta_p$'s. That is,
$$\{\tau_p\}_{p\in\Lambda_1}\vee\{\phi_p\}_{p\in \Lambda_2}=\{\tau_p\vee\phi_p\}_{p\in \Lambda_1\cap\Lambda_2},$$
$$\{\tau_p\}_{p\in\Lambda_1}\wedge \{\phi_p\}_{p\in \Lambda_2}=\{\psi_p\}_{p\in \mathcal{S}\{\Lambda_1,\Lambda_2\}},$$
where $\mathcal{S}\{\Lambda_1,\Lambda_2\}$ is the ideal of $G_0^+$ generated by $\Lambda_1$ and $\Lambda_2,$ and $\psi_p$ will be defined later.
Let $\Lambda=\Lambda_1\cap\Lambda_2$. 
Now let's define a map $\widehat{\alpha}:\{\text{Aff}_b(\Delta_p)^+\}_{p\in S(\Lambda_1,\Lambda_2)}\rightarrow\mathbb{R}^+ $ by the following:

If $f\in \text{Aff}_b(\Delta_p)^+$ for some $p\in \Lambda_1$,  let
 $$\widehat{\alpha}(f)=\inf\{f_1(\tau_{p_1})+f_2(\tau_{p_2}\wedge\phi_{p_2}): f=\lambda^*_{p,p_1}(f_1)+\lambda^*_{p,p_2}(f_2),f_i\in\text{Aff}(\Delta_{p_i})^+,p_2\in \Lambda\}.$$
 
 If  $f\in \text{Aff}_b(\Delta_p)^+$ for some $p\in \Lambda_2$,  let
 $$\widehat{\alpha}(f)=\inf\{f_1(\phi_{p_1})+f_2(\tau_{p_2}\wedge\phi_{p_2}): f=\lambda^*_{p,p_1}(f_1)+\lambda^*_{p,p_2}(f_2),f_i\in\text{Aff}(\Delta_{p_i})^+,p_2\in \Lambda\}.$$
 
  If $f\in \text{Aff}(\Delta_p)^+$ for some $p\in S(\Lambda_1,\Lambda_2)$, let\\
    $\widehat{\alpha}(f)= \inf\{\widehat{\alpha}( f_1)+\widehat{\alpha}(f_2): f=\lambda_{p,p_1}^*(f_1)+\lambda_{p,p_2}^*(f_2),~f_i\in \text{Aff}(\Delta_{p_i})^+, p_i\in \Lambda_i, i=1,2\}. $


Claim: $\widehat{\alpha}(\lambda_{p,q}^*(g))=\widehat{\alpha}(g),$ for all $q\leq p\in S(\Lambda_1, \Lambda_2)$ and $g\in \text{Aff}_b(\Delta_q).$\\
It is enough to prove the above equation for $p\in \Lambda_1.$
We know that\\
 $\widehat{\alpha}(g)=\inf\{f_1(\tau_{q'})+f_2(\tau_e\wedge\phi_e): g=\lambda^*_{q,q'}(f_1)+\lambda^*_{q,e}(f_2),~f_2\in\text{Aff}(\Delta_e)^+,~e\in\Lambda, f_1\in\text{Aff}(\Delta_{q'})^+,~q'\leq q\}$ and \\
  $\widehat{\alpha}(\lambda^*_{p,q}(g))=\inf\{g_1(\tau_{p'})+g_2(\tau_{e'}\wedge\phi_{e'}): \lambda^*_{p,q}(g)=\lambda^*_{p,p'}(g_1)+\lambda^*_{p,e'}(g_2),e'\in\Lambda, g_1\in\text{Aff}(\Delta_{p'})^+,g_2\in\text{Aff}(\Delta_{e'})^+,~p'\leq p\}.$\\
  For any $g=\lambda^*_{q,q'}(f_1)+\lambda^*_{q,e}(f_2)$, we have
  $\lambda^*_{p,q}(g)=\lambda^*_{p,q'}(f_1)+\lambda^*_{p,e}(f_2)$.
  Thus, $$\widehat{\alpha}(\lambda^*_{p,q}(g))\leq\widehat{\alpha}(g).$$
  On the other hand, if 
  $\lambda^*_{p,q}(g)=\lambda^*_{p,p'}(g_1)+\lambda^*_{p,e'}(g_2)$, then 
  $$\lambda^*_{p,q}(g)\geq \lambda^*_{p,p'}(g_1) \mbox{ and }\lambda^*_{p,q}(g)\geq \lambda^*_{p,e'}(g_1).$$
By the hereditary property, there exist $h_1,~h_2\in\text{Aff}_b(\Delta_q)$ such that 
$$\lambda^*_{p,p'}(g_1)=\lambda^*_{p,q}(h_1)~~\mbox{ and }~~\lambda^*_{p,e'}(g_2)=\lambda^*_{p,q}(h_2).$$That is, 
$$\lambda^*_{p,q}(g)=\lambda^*_{p,q}(h_1)+\lambda^*_{p,q}(h_2).$$
Therefore, $g-h_1-h_2\in \ker(\lambda_{p,q}^*).$
Since any element in $\ker (\lambda^*_{p,q})$ can be written as a difference of two elements in $\ker (\lambda^*_{p,q}) \cap \text{Aff}(\Delta_q)^+$,
there are $j_1,j_2\in\ker (\lambda^*_{p,q}) \cap \text{Aff}(\Delta_q)^+$ such that
$$g+j_2=h_1+h_2+j_1.$$
By redefine $h_1$ as $h_1+j_1,$ we have $g\leq h_1+h_2.$
Therefore,
 $\widehat{\alpha}(\lambda^*_{p,q}(g))\geq\widehat{\alpha}(g)$ and the 
claim is true.

Then $\widehat{\alpha}$ is well-defined.
It is well-known that there is a one-to-one correspondence between the elements in $\Delta_p$ and the affine maps from $\text{Aff}_b(\Delta_p)^+$ to $\mathbb{R}^+$.
If the image of $\widehat{\alpha}$ is finite when restricted on $\text{Aff}_b({\Delta_p})^+$ for some $p,$ then there exists an element in $\Delta_p$ (we denote it by $\alpha_p$) such that
$\widehat{\alpha}(f)=f(\alpha_p)$ for all $f\in\text{Aff}({\Delta_p})^+$.

It is routine to check $$\widehat{\alpha}(f)=t\widehat{\alpha}(g_1)+(1-t)\widehat{\alpha}(g_2),$$ for all $t\in [0,1]$, $f,~g_i\in \text{Aff}_b(\Delta_{p})^+$ and $p\in G_0^+$ satisfying
 $f=tg_1+(1-t)g_2$.

Finally, for $p\in S(\Lambda_1,\Lambda_2)$, let $\psi_p=\alpha_p.$
If $q\leq p$, then 
$$g(\psi_q)=g(\alpha_q)=\lambda_{p,q}^*(g)(\alpha_p)=\lambda_{p,q}^*(g)(\psi_p)=g(\lambda_{p,q}(\psi_p))$$
for all $g\in \text{Aff}(\Delta_q)^+$. 
That is, $\psi_q$ and $\lambda_{p,q}(\psi_p)$ correspond to the same map
from $\text{Aff}(\Delta_q)$ to $\mathbb{R}^+$.
Therefore, $\psi_q=\lambda_{p,q}(\psi_p)$.
Thus, $\{\psi_p\}_{p\in S(\Lambda_1,\Lambda_2)}$ is an element in $X$.

Define  $s^G: G_0^+\rightarrow \text{Aff}(X)^+$ as follows:
$$
s^G(q)(\{\tau_p\}_{p\in \Lambda}) = \left\{
  \begin{array}{l l}
     s^G_q(q)(\tau_q) & \quad \text{if $q\in \Lambda$}\\
     \infty & \quad \text{otherwise}
   \end{array} \right.
$$
for all $q\in G_0^+.$ 

It is easy to see that $((G_0,G_0^+),\Sigma G, G_1, X)$ is an object in $\mathcal{E}.$
Let $$\mathcal{F}:\{\mbox{Objects in }\mathcal{S}\}\rightarrow\{\mbox{Objects in }\mathcal{E}\}$$ be the map defined by
sending
$$((G_0,G_0^+),\Sigma G, G_1, \{\Delta_p^G\}_{p\in G_0^+})\mapsto((G_0,G_0^+),\Sigma G, G_1, X),$$
which completes the proof.

\end{proof}





\begin{lemma}\label{close} Let $a,b\in A_+$ be such that $||a-b||<\varepsilon$. Then $(a-\varepsilon)_+\preccurlyeq b.$
\end{lemma}
\begin{proof} By Lemma 2.2 of \cite{KR},  there is $d\in A$ with $||d||\leq 1$ and $(a-\varepsilon)_+=dbd^*$. Hence, $(a-\varepsilon)_+\sim b^{1/2}d^*db^{1/2}\leq b.$
\end{proof}
\begin{lemma}\label{proj} Let $A$ be a $\text{C}^{*}$-algebra and $\tau$ be a lower semicontinuous trace on $A$.
Let $A^{\tau}$ be the ideal of $A$ defined  in Proposition \ref{fin}.
If $p\in \overline{{A^{\tau}}}$ (the closure of $A^{\tau}$) is a projection, then $\tau(p)<\infty.$
\end{lemma}
\begin{proof} If $p\in A^{\tau}$, then $\tau(p)<\infty$ by definition.
If $p\in\overline{A^{\tau}}\setminus A^{\tau}$, then there exists a sequence $\{x_i\}_{i=1}^{\infty}\subseteq A_+^{\tau}$ with $\lim\limits_{i\rightarrow\infty}x_i=p$. For any $\varepsilon>0$, there exists $k\in\mathbb{N}$ such that:
$$||x_k-p||<\varepsilon.$$ 
By Lemma \ref{close} $(p-\varepsilon)_+\preccurlyeq x_k$.
Since $(p-\varepsilon)_+=(1-\varepsilon)p$, 
$$\tau((1-\varepsilon)p)=\tau((p-\varepsilon)_+)\leq\tau( x_k)<\infty.$$ 
Therefore, $\tau(p)<\infty.$

\end{proof}

\begin{lemma}\label{equa} Let $\phi$ and $\phi'$ be two lower semi-continuous traces on a $\text{C}^*$-algebra $A$ and $e$ be a projection in $A$. 
Suppose that both $\phi$ and $\phi'$ are finite when restricted on $(\overline{eAe})_+$ and 
$$\phi(x)=\phi'(x) \mbox{ for all } x\in (\overline{eAe})_+.$$ Then $$\phi(x)=\phi'(x) \mbox{ for all } x\in (\overline{AeA})_+,$$ where $\overline{AeA}$ stands for the closed two-sided ideal generated by $e$.
\end{lemma}

\begin{proof} For $x\in A_+$, let $$\psi(x)=\sup\{\phi(y):~y\preccurlyeq x,y\in (\overline{eAe})_+\}.$$
By Theorem \ref{extend}, $\psi$ is a lower semi-continuous trace on $A$ and $\phi(x)=\psi(x)$ for all $x\in (\overline{eAe})_+$.
We only need to show that $\phi(x)=\psi(x)$ for all $x\in(\overline{AeA})_+$.

For any $x\in(\overline{AeA})_+,$ $y\in (\overline{eAe})_+$ satisfying $y\preccurlyeq x$, we have $$\psi(y)=\phi(y)\leq \phi(x).$$
 Taking supremum on both sides for all such $y$ we get 
\begin{equation}\label{eq}
\psi(x)\leq \phi(x) \text{ for all }x\in(\overline{AeA})_+.
\end{equation}
 Let $\Omega=\{\sum_{k=1}^n a_keb_k:a_k,~b_k\in A,~n\in\mathbb{Z}\}$ be a subset of $\overline{AeA}$. 
Claim:$$\phi(x)=\psi(x)\mbox{ for all }x\in \Omega_+.$$
In fact, for $x=aeb\in A_+$ with $a,b\in A$, we have
$$\phi((ae)^*(ae))=\phi(ea^*ae)<\infty$$ and $$\phi((eb)^*(eb))=\phi((eb)(eb)^*)=\phi(ebb^*e)<\infty.$$
Similarly, $$\psi((ae)^*(ae))=\psi(ea^*ae)<\infty,$$
 $$\psi((eb)^*(eb))=\psi((eb)(eb)^*)=\psi(ebb^*e)<\infty.$$
Therefore, by  Proposition \ref{comm}, $$\phi(x)=\phi(aeeb)=\phi(ebae)=\psi(ebae)=\psi(aeb)=\psi(x).$$
Thus, the claim is true.

Let $x$ be a positive element in $\overline{AeA}$ and $\{y_n\}$ be an increasing sequence in $\Omega_+$ with $\lim\limits_{n\rightarrow\infty}y_n= x$.
Then by the above claim $$\phi(x)\leq\lim\phi(y_n)=\lim\psi(y_n)\leq\psi(x),$$ where the first inequality is due to $\phi$ being lower semi-continuous.

Therefore, combined with the inequality (\ref{eq}), we get $\phi(x)=\psi(x)$ for all $x\in (\overline{AeA})_+$.
That is, $$\phi(x)=\sup\{\phi(y):~y\preccurlyeq x,y\in (eAe)_+\} \mbox{ for all } x\in (\overline{AeA})_+.$$
Similarly, $$\phi'(x)=\sup\{\phi'(y):~y\preccurlyeq x,y\in (eAe)_+\} \mbox{ for all } x\in (\overline{AeA})_+.$$
Therefore,
$$\phi'(x)=\phi(x) \mbox{ for all } x\in (\overline{AeA})_+.$$

\end{proof}

\begin{lemma}\label{wd}
Let $A$ be a $\text{C}^*$-algebra with the ideal property. Suppose that for any projection $e\in A$, there is a lower semi-continuous trace $\phi_e$ on
$\overline{AeA}$ satisfying
$$\phi_{e_1}(x)=\phi_{e_2}(x) \mbox{ for all } x\in (\overline{Ae_1A})_+\cap (\overline{Ae_2A})_+,$$ where $e_1,~e_2$ are two projections in $A$.
Then there is a lower semi-continuous trace $\phi$ on $A$ such that $$\phi(x)=\phi_e(x)$$
 for any $x\in (\overline{AeA})_+$ and any projection $e\in A$.
\end{lemma}
\begin{proof}
If $A$ is generated by a single projection, say $e$, then let $\phi=\phi_e$. 
In this case, if $e'$ is any projection in $A$, then $$\phi(x)=\phi_e(x)=\phi_{e'}(x) \mbox{ for all } x\in \overline{Ae'A}.$$ 

If $A$ is not generated by a single projection, let $J$ be an ideal of $A$ and $\phi_J$ be a lower semi-coutinuous trace on $J$ satisfying desired properties, 
i.e. $$\phi_J(x)=\phi_e(x)\mbox{ for all }x\in (\overline{AeA})_+,$$ where $e$ is any projection in $J$. 
Let $p$ be a projection outside $J$ and $J'$ be the closed ideal generated by $J$ and $p.$
Let $\phi_{J'}$ be an extension of $\phi_J$ on $J'$ defined as follows.

Since $(J')_+=J_++(\overline{ApA})_+, $ for any $x\in (J')_+$, there exist $x_1\in J_+$ and $x_2\in (\overline{ApA})_+$ such that $x=x_1+x_2.$
Define $$\phi_{J'}(x)=\phi_{J}(x_1)+\phi_{p}(x_2).$$
If there exist  $y_1\in J_+$ and $y_2\in (\overline{ApA})_+$ such that $x_1+x_2=x=y_1+y_2 $, then 
$$x_1-y_1=y_2-x_2\in J_+\cap(\overline{ApA})_+.$$
For any projection $e\in J\cap(\overline{ApA})$, we have
$$\phi_J(z)=\phi_e(z)=\phi_p(z)$$
for all $z\in\overline{AeA}.$
Since $A$ has the ideal property, $J\cap(\overline{ApA})$ is generated by projections.
Therefore, 
$\phi_J(x_1-y_1)=\phi_p(y_2-x_2)$. That is $$\phi_J(x_1)+\phi_p(x_2)=\phi_J(y_1)+\phi_p(y_2).$$
Thus, $\phi_{J'}$ is well-defined. It is obvious that $\phi_{J'}$ is lower semi-continuous trace on $J'$ satisfying desired properties.

To complete the proof, we apply Zorn's lemma.
Let X be the set of all pairs
$$(J, \phi_J)$$
where $J$ is a sub-ideal of I and $\phi_J$ is a lower semi-continuous trace satisfying 
the desired properties. 
Define the relation $<$ on $X$ by
$$(J_1, \phi_{J_1}) <(J_2, \phi_{J_2})\mbox{ if and only if  }J_1\subseteq J_2\mbox{ and } \phi_{J_2}|_{J_1}= \phi_{J_1}.$$ This is a partial ordering with the property that
any totally ordered subset has a maximal element. Zorn's lemma says
that $X$ has a maximal element, say $(\widehat{J}, \phi_{\widehat{J}})$. 
If $\widehat{J}$ is a proper subspace of A,
i.e., $\widehat{J}\neq A$, then the argument given before produces an extension of $\phi_{\widehat{J}}$ to a
larger ideal, contradicting the maximality of $(\widehat{J}, \phi_{\widehat{J}})$.



\end{proof}

\begin{corollary}\label{welF}
Let $\mathcal{I}$ be the class of all stably finite $\text{C}^*$-algebras satisfying the ideal property.
Then
$$\mathcal{E}(A)\cong\mathcal{F}(\mathcal{S}(A))$$
for all $A\in\mathcal{I}.$
\end{corollary}

\begin{proof}
Let $A$ be any stably finite $\text{C}^*$-algebra satisfying the ideal property.
We have
$$\mathcal{S}(A)=((\text{K}_0(A),\text{K}_0(A)^+),\Sigma A, \text{K}_1(A),\{\text{T}_{\text{F}}(\overline{pAp})\}_{p\in\text{K}_0(A)^+}).$$
Let $\mathcal{F}(\mathcal{S}(A))=((\text{K}_0(A),\text{K}_0(A)^+),\Sigma A, \text{K}_1(A),X)$, where $X$ is the collection of all sets of the form
$\{\tau_p\}_{p\in \Lambda}$ defined as in the proof of Theorem \ref{DefF}.
We will show that  $\text{T}A$ is isomorphic to $X$.

For any $\tau\in\text{T}(A)$, 
let $$A_+^{\tau}=\{x\in A: \tau(x)<\infty\}$$
 and $A^\tau$ be the linear span of $A^\tau_+$, which is an ideal of $A$ by Lemma \ref{fin}.
Let $$\mathcal{P}_{\tau}:=\{p\in \overline{A_+^\tau}:p \mbox{ is a projection}\},$$
By Lemma \ref{proj}, $\mathcal{P}_{\tau}\subseteq A_+^{\tau}$. 
Since $A$
is a $\text{C}^*$-algebra with the ideal property, $\overline{A_+^{\tau}}$ is generated by $\mathcal{P}_{\tau}$. 
Thus, $\mathcal{P}_{\tau}$ is not empty.  
For any $f\in\mathcal{P}_{\tau}$, we have  $\tau|_{\overline{fAf}}\triangleq\tau_f$ is a finite trace on $\overline{fAf}$.
Let $\Lambda_{\tau}$ be the sub-semigroup of $\text{K}_0(A)^+$ generated by $\mathcal{P}_{\tau}$.
For each $p\in\text{K}_0(A)^+$, we can define $\tau_p$ by extending of $\tau_f$'s for $f\in\mathcal{P}_{\tau}$ to matrix algebra.
Therefore, $\{\tau_f\}_{f\in \text{K}_0(A)^+}$ is an element in $X$.

Define $\beta: \text{T}(A)\rightarrow X$ by sending $\tau$ to $\{\tau_f\}_{f\in \text{K}_0(A)^+}$. 
It is obvious to see  that $\beta$ is injective since $A$ has the ideal property. 

Let $\{\tau_p\}_{p\in \Lambda}$ be any element in $X$.
By the definition of $X$,  $\tau_p$ is a finite trace on $\overline{pAp}$ for each $p$.
By Theorem \ref{extend}, for each $p$, we can extend $\tau_p$ to be a lower semi-continuous trace on the ideal $\overline{ApA}$, still denoted by $\tau_p.$

Claim: $\tau_{p_1}(x)=\tau_{p_2}(x)\mbox{ for all } x\in(\overline{Ap_1A})\cap (\overline{Ap_2A}),$
where $p_1,p_2$ are in $\Lambda.$ \\
In fact, since $A$ has the ideal property, it is enough to show 
$$\tau_{p_1}(x)=\tau_{p_2}(x)\mbox{ for all } x\in\overline{AqA}, $$ where $q$ is any projection in $(\overline{Ap_1A})\cap (\overline{Ap_2A}).$
By Lemma \ref{lin}, there are integers $m,~n$ such that $q\leq np_1$ and $q\leq mp_2.$ Thus, $$\tau_{p_1}(x)=\tau_{np_1}(x)=\tau_q(x)=\tau_{mp_2}(x)=\tau_{p_2}(x).$$
Therefore, the claim is true.

Applying Lemma \ref{wd}, we know that there is a lower semi-continuous trace $\widehat{\tau}$ on the ideal of $A$ generated by $\Lambda\cap A.$
Let $$
\tau(x) = \left\{
  \begin{array}{l l}
     \widehat{\tau}(x) & \quad \text{if $x$ is in the ideal generated by $\Lambda,$}\\
     \infty & \quad \text{otherwise}.
   \end{array} \right.
$$
Then $\tau$ is a lower semi-continuous trace on $A$ satisfying 
$$\beta(\tau)=\{\tau_p\}_{p\in\Lambda}.$$
Therefore, $\beta$ is also surjective. 
Thus, $\text{T}A$ is isomorphic to $X$.

\end{proof}

\begin{remark}
Theorem \ref{mainthm1} follows from Corollaries \ref{welG} and \ref{welF} immediately.
\end{remark}









\section{Main Theorem}
In this section, we extend the maps defined in section 3 to be  functors between 
two sub-categories of $\mathcal{S}$ and $\mathcal{E}$ and prove  Theorem 1.2.
Finally, we show that there are $\text{C}^*$-algebras without the ideal property whose Elliott invariant cannot be derived from the Stevens invariant.

\begin{definition}\label{compatible} Let $A,B$ be two $C^*$-algebras. Let $\alpha$ from $\mbox{K}_0(A)$
to $\text{K}_0(B)$ be a homomorphism, and $\xi:\text{T}B\rightarrow \text{T}A$ be an affine map. We say that $\alpha$ and $\xi$ are compatible if $$\tau(\alpha(x))=(\xi(\tau))(x)$$ for all $x\in\text{K}_0(A)_+$ and $\tau\in \text{T}B$.
\end{definition}

\begin{proposition}
Let $A,B$ be two $\text{C}^*$-algebras. If  there is a morphism 
$$\Theta=(\theta_0,\theta_1,\{\xi^p\}_{p\in \text{K}_0(A)^+}):\mathcal{S}(A)\rightarrow\mathcal{S}(B),$$
then $\theta_0$ and $\xi^{e}$ are compatible for all $e\in\text{K}_0(A)^+$.
\end{proposition}
\begin{proof}
See 1.11 in \cite{JJ}.
\end{proof}

\begin{lemma}\label{lemma1} Let $A$ be a $\text{C}^{*}$-algebra with the ideal property and $B$ be a $\text{C}^*$-algebra satisfying \\
$(1)$ There exists a scaled ordered isomorphism
$\alpha: \text{K}_0A\rightarrow \text{K}_0 B$;\\
$(2)$ There is an isomorphism $\xi: \text{T}B\rightarrow \text{T}A$ which is compatible with $\alpha$. \\
Let $\tau\in \text{T}B$
and $I$ be the closed ideal of $A$ generated by the set $$\{e\in \mathcal{P}(A):~\tau(\alpha(e))<\infty\}.$$
Then $$\xi(\tau)(x)=+\infty \text{ for all } x\in A_+\backslash I_+.$$
\end{lemma}

\begin{proof} Let $\phi=\xi(\tau).$ For any projection $p\in\overline{A^{\phi}}$, by Lemma \ref{proj}, $\phi(p)<+\infty$. Since $\xi$ is compatible with $\alpha$,  $$\tau(\alpha(p))=\xi(\tau)(p)=\phi(p)<+\infty. $$
That is $\alpha(p)\in B^{\tau}_+$ and  $p\in I$. 
Therefore $\mathcal{P}(\overline{A^{\phi}})\subseteq \mathcal{P}(I)$. 
Since $A$ is a $\text{C}^*$-algebra with the ideal property, $\overline{A^{\phi}}\subseteq I$.  So $$A\setminus I\subseteq A\setminus \overline{A^{\phi}}.$$ Hence $\phi(x)=+\infty$ for all $x\in A_+\setminus I_+.$

\end{proof}

The following lemma is well-known  (see Lemma 3.3.6 of \cite{Lin book}).
\begin{lemma}\label{lin}
If $p$ is a projection in $A$, $b\in A_+$ and $p$ is in the ideal generated by $b$, then there are $x_1,x_2,\cdots,x_k\in A$ such that $p=\sum\limits_{i=1}^kx_ibx_i^*$.
\end{lemma}

The following theorem is the main theorem of this section. 
\begin{theorem}\label{isom}
Let $\mathcal{E}_{\mathcal{I}}$ and $ \mathcal{S}_{\mathcal{I}}$ be two sub-categories of $\mathcal{E}$ and $\mathcal{S}$, respectively, where
$\mathcal{I}$ is the class of $\text{C}^*$-algebras with the ideal property.
Then $\mathcal{E}_{\mathcal{I}}$ and $ \mathcal{S}_{\mathcal{I}}$ are isomorphic. That is,
there are canonical functors $$\mathcal{F}:\mathcal{S}_{\mathcal{I}}\rightarrow \mathcal{E}_{\mathcal{I}},\quad
\mathcal{G}: \mathcal{E}_{\mathcal{I}}\rightarrow\mathcal{S}_{\mathcal{I}}$$ such that 
$$\mathcal{G}\circ \mathcal{F}=Id_{\mathcal{S}_{\mathcal{I}}},\quad\quad \mathcal{F}\circ \mathcal{G}=Id_{\mathcal{E}_{\mathcal{I}}}.$$
\end{theorem}

\begin{proof} We divide the proof into three steps.

\underline{\textbf{Step I: Construction of the functor $\mathcal{F}:\mathcal{S}_{\mathcal{I}}\rightarrow \mathcal{E}_{\mathcal{I}}$.}}

Let $\mathcal{F}$ map the objects in $\mathcal{S}_{\mathcal{I}}$ to the objects in $\mathcal{E}_{\mathcal{I}}$ be defined by
$$\mathcal{F}(\mathcal{S}(\bullet))=\mathcal{E}(\bullet).$$
By Corollary \ref{welF}, this map is well-defined.
Let $A,B\in\mathcal{I}$ such that there is an arrow $\Phi=(\theta_0,\theta_1, \{\xi^e\}_{e\in \text{K}_0(A)^+})$  from $\mathcal{S}(A)$
to $\mathcal{S}(B)$.
We need to construct an affine map $\zeta$ from $\text{T}B$ to $\text{T}A$ which is  compatible with $\theta_0$.

For $\tau \in \text{T}B$, let $$B_+^{\tau}=\{x\in B: \tau(x)<\infty\}$$
 and $B^\tau$ be the linear span of $B^\tau_+$, which is an ideal of $B$ by Lemma \ref{fin}.
Let $$\mathcal{P}_{\tau}:=\{p\in \overline{B_+^\tau}:p \mbox{ is a projection}\},$$
$$\Lambda_{\tau}=\{f\in\mathcal{P}_{\tau}:f \text{ is in the image of }\theta_0\}.$$
Then $\Lambda_{\tau}\subseteq \mathcal{P}_{\tau}\subseteq B_+^{\tau}$, where the last inclusion is by Lemma \ref{proj}.

For any $f\in\Lambda_{\tau}$, we have $\tau(f)<\infty$ and $\tau|_{\overline{fBf}}\triangleq\tau_f$ is a finite trace on $\overline{fBf}$.
Let $e\in A$ be a pre-image of $f$ under $\theta_0$. Define  $\phi_{e}=\xi^e(\tau_f)$.
Since $\xi^e$ and $\theta_0|_{\overline{eAe}}$ are compatible, $\phi_e$ is a finite trace on $\overline{eAe}$.
We can extend $\phi_e$ to a lower semi-continuous trace (still denoted by $\phi_e$) on $\overline{AeA}$ by
$$\phi_e(x):=\sup\{\phi_e(y):y\preccurlyeq x,~y\in (\overline{eAe})_+\}, ~~\mbox{ for all } x\in (\overline{AeA})_+.$$

Claim: If $x\in J_0:=(\overline{Ae_1A})\cap(\overline{Ae_2A})$ is a positive element, where $e_1,~e_2$ are two projections in $A$
 satisfying $\theta_0(e_i)=f_i$ for some $f_i\in\Lambda_{\tau}$, $i=1,2,$ then $$\phi_{e_1}(x)=\phi_{e_2}(x).$$

In fact,
since $A$ is a $\text{C}^*$-algebra with the ideal property, $J_0$ is generated by projections inside it.
Let $I_0$ be a closed ideal of $B$ generated by the set $\{q: q=\theta_0(p) \mbox{ for some   }p\mbox{ in } J_0\}$.
 Then we have $$I_0=(\overline{Bf_1B})\cap(\overline{Bf_2B})~~~\mbox{ and }~~~~\tau_{f_1}|_{I_0}=\tau_{f_2}|_{I_0}.$$
 Let $p$ be any projection in $J_0$ and let $q=\theta_0(p)$.
By Lemma \ref{lin}, there exist natural numbers $n_1,n_2$ such that $p\leq n_ie_i$ for $i=1,2$.
 Therefore, by the compatible condition, the following diagrams are commutative:
$$\xymatrix@C=1cm{{\rm{T}}_{\rm{F}}(\overline{\widetilde{f}_1B\widetilde{f}_1})\ar[d]^i\ar[r]^{\xi^{\widetilde{e}_1}}&{\rm{T}}_{\rm{F}}
(\overline{\widetilde{e}_1A\widetilde{e}_1})\ar[d]^{i}\\
{\rm{T}}_{\rm{F}}(\overline{qBq})\ar[r]^{\xi^{p}}&{\rm{T}}_{\rm{F}}(\overline{pAp})},$$
$$\xymatrix@C=1cm{{\rm{T}}_{\rm{F}}(\overline{\widetilde{f}_2B\widetilde{f}_2})\ar[d]^i\ar[r]^{\xi^{\widetilde{e}_2}}&{\rm{T}}_{\rm{F}}(\overline{\widetilde{e}_2A\widetilde{e}_2})\ar[d]^{i}\\
{\rm{T}}_{\rm{F}}(\overline{qBq})\ar[r]^{\xi^{p}}&{\rm{T}}_{\rm{F}}(\overline{pAp})},$$ where $\widetilde{f}_i=n_i f$ and $\widetilde{e}_i=n_i e_i$ for $i=1,2$.
Since $${\rm{T}}_{\rm{F}}(\overline{\widetilde{f}_iB\widetilde{f}_i})={\rm{T}}_{\rm{F}}(\text{M}_{n_i}(\overline{f_iBf_i}))={\rm{T}}_{\rm{F}}(\overline{f_iBf_i}),$$
$${\rm{T}}_{\rm{F}}(\overline{\widetilde{e}_iA\widetilde{e}_i})={\rm{T}}_{\rm{F}}(\text{M}_{n_i}(\overline{e_iAe_i}))={\rm{T}}_{\rm{F}}(\overline{e_iAe_i}),$$ we can get the following commutative diagrams:
 $$\xymatrix@C=1cm{{\rm{T}}_{\rm{F}}(\overline{f_1Bf_1})\ar[d]^i\ar[r]^{\xi^{e_1}}&{\rm{T}}_{\rm{F}}(\overline{e_1Ae_1})\ar[d]^{i}\\
{\rm{T}}_{\rm{F}}(\overline{qBq})\ar[r]^{\xi^{p}}&{\rm{T}}_{\rm{F}}(\overline{pAp})},$$
$$\xymatrix@C=1cm{{\rm{T}}_{\rm{F}}(\overline{f_2Bf_2})\ar[d]^i\ar[r]^{\xi^{e_2}}&{\rm{T}}_{\rm{F}}(\overline{e_2Ae_2})\ar[d]^{i}\\
{\rm{T}}_{\rm{F}}(\overline{qBq})\ar[r]^{\xi^{p}}&{\rm{T}}_{\rm{F}}(\overline{pAp})}.$$ 
Therefore, $$i\circ \xi^{e_1}(\tau_{f_1})=\xi^{p}\circ i(\tau_{f_1}),\quad i\circ \xi^{e_2}(\tau_{f_2})=\xi^{p}\circ i(\tau_{f_2}).$$
 That is 
 $$\phi_{e_1}|_{\overline{pAp}}=\xi^{p}(\tau_{f_1}|_{\overline{qBq}}),\quad\phi_{e_2}|_{\overline{pAp}}=\xi^{p}(\tau_{f_2}|_{\overline{qBq}}).$$
Since $\tau_{f_1}|_{\overline{qBq}}=\tau_{f_2}|_{\overline{qBq}}$, $\phi_{e_1}|_{\overline{pAp}}=\phi_{e_2}|_{\overline{pAp}}$. 
By Theorem \ref{equa}, we have  $$\phi_{e_1}|_{\overline{ApA}}=\phi_{e_2}|_{\overline{ApA}}.$$ Therefore, $\phi_{e_1}|_{J_0}=\phi_{e_2}|_{J_0}$ since $J_0$ is generated by projections inside it.
 Thus, we have proved the claim.

Let $I$ be the ideal generated by  $\theta_0^{-1}(\Lambda_{\tau})$ (the preimage of $\Lambda_{\tau}$).
Applying Lemma \ref{wd}, there is a lower semi-continuous trace $\phi'$ on $I$ satisfying certain properties.
Extend $\phi'$ to a lower semi-continuous trace on $A$  (denoted by $\phi$) as follows:
$$
\phi(x) = \left\{
  \begin{array}{l l}
     \phi'(x) & \quad \text{if $x\in I,$}\\
     \infty & \quad \text{otherwise}.
   \end{array} \right.
$$
 Let $\zeta(\tau)=\phi.$ Then it remains to check that $\zeta $ is an affine map compatible with $\theta_0$.

For any $\tau\in \text{T}B$, $e\in\text{K}_0(A)_+$, let $f=\theta_0(e)$ and $\zeta(\tau)=\phi.$
If $\tau(f)<\infty$, then $$\zeta(\tau)(e)=\phi(e)=\phi_e(e)=\xi^*(\tau_f)(e)=\tau_f(\theta_0(e))=\tau(\theta_0(e)).$$
 If $\tau(f)=\infty$, then $e$ is not in the preimage of $\Lambda_{\tau}$ under $\theta_0.$ Thus, $$\zeta(\tau)(e)=\phi(e)=\infty=\tau(f).$$
 Therefore, $\zeta(\tau)(e)=\tau(\theta_0(e))$ for all $\tau\in \text{T}B$ and $e\in \text{K}_0(A)_+.$ 
That is, $\zeta$ is compatible with $\theta_0.$

 Let $\rho',~\rho''\in \text{T}B$ and $\rho=t\rho'+(1-t)\rho''$ for $t\in(0,1)$.
 Let $\mu=\zeta(\rho),$ $\mu'=\zeta(\rho')$, $\mu''=\zeta(\rho'')$.
 If $x$ is a positive element in the ideal generated by the preimage of $\Lambda_{\rho}$ under $\theta_0,$ (without loss of generality, we can assume $x\in \overline{AeA}$ and $\theta_0(e)=f$), then by the above construction, we have 
$$\zeta(\rho)(x)=\mu(x)=\mu|_{\overline{AeA}}(x):=\mu_e(x)=\xi^e(\rho_f)(x).$$
Since $\xi^e$ is an affine map, we have 
\begin{align*}
\xi^e(\rho_f)(x)
&=\xi^e(t\rho'_f+(1-t)\rho''_f)(x)\\
&=t\xi^e(\rho'_f)(x)+(1-t)\xi^e(\rho''_f)(x)\\
&=t\mu'_e(x)+(1-t)\mu''_e(x)\\
&=[t\zeta(\rho')+(1-t)\zeta(\rho'')](x).
\end{align*}
 If $x\in A_+$ is not in the ideal generated by the preimage of $\Lambda_{\rho}$,  by Lemma \ref{lemma1}, $$\zeta(\rho)(x)=\mu(x)=\infty.$$
 Since $\Lambda_{\rho}=\Lambda_{\rho'}\cap\Lambda_{\rho''},$ $x$ is not in the ideal generated by 
 $\theta_0^{-1}(\Lambda_{\rho'})\cap\theta_0^{-1}(\Lambda_{\rho''})$, where 
 $\theta_0^{-1}(\cdot)$ means the preimage set under $\theta_0.$
Thus, $$\zeta(\rho')(x)=\infty\mbox{ or }\zeta(\rho'')(x)=\infty.$$
Therefore, $$\zeta(\rho)(x)=t\zeta(\rho')(x)+(1-t)\zeta(\rho'')(x)$$ for all $x\in A_+$ and all $t\in(0,1)$.
As a consequence, $\Psi=(\theta_0,\theta_1,\zeta)$ is a morphism from $\mathcal{E}(A)$ to $\mathcal{E}(B)$ in $\mathcal{E}$.
Define $$\mathcal{F}(\Phi)=\Psi.$$
(1) It is obvious that $\mathcal{F}$ maps objects to objects and morphisms to morphisms.\\
(2) Let $A$ be a $\text{C}^*$-algebra  with the ideal property. For any $ \mu\in \text{T}A$, let $J$ be the ideal generated by projections in $\Lambda_{\mu}$.
Let $\mu'$ be a lower semi-continuous trace on $A$ defined by
$$
\mu'(x) = \left\{
  \begin{array}{l l}
     \mu(x) & \quad \text{if $x\in J,$}\\
     \infty & \quad \text{otherwise}.
   \end{array} \right.
$$
By Lemma \ref{lemma1}, $\mu'(x)=\mu(x)$ for all $x$ in $A$. Therefore,
$$\mathcal{F}(id_{\mathcal{S}(A))})=id_{\mathcal{E}(A)}=id_{\mathcal{F}(\mathcal{S}(A))}.$$
(3) Let $A_1,~A_2,~A_3$ be $\text{C}^*$-algebras with the ideal property. Suppose $\Psi_1:\mathcal{S}(A_1))\rightarrow \mathcal{S}(A_2)$, $\Psi_2:\mathcal{S}(A_2)\rightarrow \mathcal{S}(A_3)$
are two morphisms. By the functoriality of $\text{K}_0$ and $\text{T}$, we can get $$\mathcal{F}(\Psi_2\circ\Psi_1)=\mathcal{F}(\Psi_2)\circ\mathcal{F}(\Psi_1).$$

Therefore, $\mathcal{F}$ is a functor from the $\mathcal{S}_{\mathcal{I}}$ to $\mathcal{E}_{\mathcal{I}}$.\\



\underline{\textbf{Step II: Construction of the map $\mathcal{G}:\mathcal{E}_{\mathcal{I}}\rightarrow \mathcal{S}_{\mathcal{I}}$.}}

Let $\mathcal{G}$ map the objects in $\mathcal{E}_{\mathcal{I}}$ to the objects in $\mathcal{S}_{\mathcal{I}}$ be defined by
$$\mathcal{G}(\mathcal{E}(\bullet))=\mathcal{S}(\bullet).$$
By Corollary \ref{welG}, this map is well-defined.
Let $A,B\in\mathcal{I}$ such that there is an arrow $\Psi=(\theta_0,\theta_1, \zeta)$  from $\mathcal{E}(A)$
to $\mathcal{E}(B)$.
We need to construct an arrow from $\mathcal{S}(A)$ to $\mathcal{S}(B)$.
Let $e$ be any projection in $A$. First we want to construct an affine map $\xi^e$ from $\text{T}_{\text{F}}(\overline{\theta_0(e)B\theta_0(e)})$ to $\text{T}_{\text{F}}(\overline{eAe})$.

Let $f=\theta_0(e)$. For $\tau_f\in \text{T}_{\text{F}}(\overline{fBf})$,  define a trace $\tau'_f$ on the closed ideal $\overline{BfB}$ by
$$\tau'_f(x)=\sup\{\tau_f(y): y\preccurlyeq x, y\in (\overline{fBf})_+\}, \mbox{ for all } x\in \overline{BfB}.$$
Then by Theorem \ref{extend}, $\tau'_f$ is lower semi-continuous.
Let ${\tau}$ be a trace on $B$ defined by
$$
{\tau}(x) = \left\{
  \begin{array}{l l}
     \tau'_f(x) & \quad \text{if $x\in \overline{BfB}$}\\
     \infty & \quad \text{otherwise}.
   \end{array} \right.
$$
Then ${\tau}$ is a lower semi-continuous trace on $B$. Let ${\phi}=\zeta({\tau})$.
By Lemma \ref{lemma1}, $${\phi}(x)=\infty \mbox{ for any } x\in A\setminus (\overline{AeA}).$$
Let $\phi'_e={\phi}|_{\overline{AeA}}$ and $\phi_e={\phi}|_{\overline{eAe}}$.
Then $$\phi_e(e)={\phi}(e)=\zeta({\tau})(e)={\tau}(\theta_0(e))=\tau(f)<\infty.$$
Thus $\phi_e$ is a finite trace on $\overline{eAe}$.
Define $\xi^e(\tau_f)=\phi_e $.
Then it is routine to check that $\xi^{e}$ is an affine map.

 Let $e'\in\mathcal{P}(A)$ and $f'\in\mathcal{P}(B)$ be such that $e'\leq e,~f'\leq f$ and $\theta_0(e')=f'$.
 Let $\tau_{f}\in {\rm{T}}_{\rm{F}}(\overline{fBf})$ be any finite trace.
 We need to show the following diagram commutes:
$$\xymatrix@C=1cm{{\rm{T}}_{\rm{F}}({\overline{fBf}})\ar[d]^i\ar[r]^{\xi^{e}}&{\rm{T}}_{\rm{F}}(\overline{eAe})\ar[d]^{i}\\
{\rm{T}}_{\rm{F}}(\overline{f'Bf'})\ar[r]^{\xi^{e'}}&{\rm{T}}_{\rm{F}}(\overline{e'Ae'})}.$$
That is, we need to show
$$i\circ\xi^{e}(\tau_f)=\xi^{e'}\circ i(\tau_f).$$
 Let $i(\tau_f)=\tau_{f'}$ and let ${\tau},{\tau'}\in \text{T}B$ be extensions of $\tau_f$ and $\tau_{f'}$ respectively defined as above.
 Let ${\phi}=\zeta({\tau})$ and ${\phi'}=\zeta({\tau'})$.
 By the definition of $\tau_{f'}$, we know that $\tau_{f'}=\tau_f|_{\overline{f'Bf'}}$.
By Theorem \ref{equa}, ${\tau'}={\tau}$ on $\overline{Bf'B}$. Thus, 
$$t{\tau}(x)+(1-t)\tau'(x)={\tau'(x)},\mbox{ for all }x\in \overline{Bf'B} \mbox{ and }0\leq t\leq1.$$
If $x\notin \overline{Bf'B}$, then ${\tau'}(x)=\infty$.
 Thus, 
$$t{\tau}(x)+(1-t)\tau'(x)={\tau'(x)},\mbox{ for all }x\notin \overline{Bf'B} \mbox{ and }0\leq t<1.$$
 Therefore, we have the following equality
 $$t{\tau}+(1-t)\tau'={\tau'},\mbox{ for any }0\leq t<1.$$
 Taking $t=1/2$ and since  $\zeta$ is an affine map, we get $$\frac{1}{2}\zeta({{\tau}})+\frac{1}{2}\zeta({\tau'})=\zeta({\tau'}).$$
  So $$\frac{1}{2}\phi|_{\overline{e'Ae'}}+\frac{1}{2}\phi'|_{\overline{e'Ae'}}=\phi'|_{\overline{e'Ae'}}.$$
  Since both $\phi|_{\overline{e'Ae'}}$ and $\phi'|_{\overline{e'Ae'}}$ are finite,
    $$i\circ\xi^{e}(\tau_f)={\phi}|_{\overline{e'Ae'}}={\phi'}|_{\overline{e'Ae'}}=\xi^{e'}\circ i(\tau_f).$$ Thus $\xi^{e}$ and $\xi^{e'}$ are compatible.
 Therefore, $\Phi=(\theta_0, \theta_1, \{\xi^e\}_{e\in\Sigma(A)})$
  is a morphism from $\mathcal{S}(A)$ to $\mathcal{S}(B)$  in $\mathcal{S}_{\mathcal{I}}$.
Define $$\mathcal{G}(\Psi)=\Phi.$$ The following properties of $\mathcal{ G}$ are obvious:\\
(1) $\mathcal{G}$ maps objects to objects and morphisms to morphisms by the above construction.\\
(2) For every $\text{C}^*$-algebra $A$ with the ideal property $$\mathcal{G}(id_{\mathcal{E}(A))})=id_{\mathcal{S}(A)}.$$\\
(3) Let  $A_1,~A_2,~A_3$ be $\text{C}^*$-algebras. Suppose that $\Phi_1:\mathcal{E}(A_1)\rightarrow\mathcal{E}(A_2)$, $\Phi_2:\mathcal{E}(A_2)\rightarrow \mathcal{E}(A_3)$
are two morphisms. Then $$\mathcal{G}(\Phi_2\circ\Phi_1)=\mathcal{G}(\Phi_2)\circ\mathcal{G}(\Phi_1).$$
Therefore, $\mathcal{G}$ is a functor from the $\mathcal{E}_{\mathcal{I}}$ to $\mathcal{S}_{\mathcal{I}}$.\\

\textbf{Step III: Check the identity of the theorem.}

(1) Let $A,~B$ be $\text{C}^*$-algebras with the ideal property. 
Suppose we have the following maps
$$\xymatrix@C=1cm{\mathcal{S}(A)\ar[d]^{\Phi=(\theta_0, \theta_1,\xi)}\ar[r]^{\mathcal{F}}&\mathcal{E}(A)\ar[d]^{\Psi=(\theta_0,\theta_1,\zeta)}\ar[r]^{\mathcal{G}}&\mathcal{S}(A)\ar[d]^{\widehat{\Phi}=({\theta_0},\theta_1,\widehat{\xi})}\\
\mathcal{S}(B)\ar[r]^{\mathcal{F}}&\mathcal{E}(B)\ar[r]^{\mathcal{G}}&\mathcal{S}(B)},$$
where $\xi=\{\xi^p\}_{p\in\Sigma(A)}$, $\widehat{\xi}=\{\widehat{\xi^p}\}_{p\in\Sigma(A)}$, 
$\mathcal{F}(\Phi)=\Psi$ and $\mathcal{G}(\Psi)=\widehat{\Phi}.$ We need to show $\widehat{\Phi}=\Phi.$

Let $\tau_f\in \text{T}_{\text{F}}(\overline{fBf})$ be any finite trace, where $f$ is a projection in $B$ with $\theta_0(e)=f$.
By the construction of step II and I, we have $$\widehat{\xi}^e(\tau_f)=\zeta(\tau)|_{\overline{eAe}}=\xi^e(\tau_f).$$
Therefore, $\xi=\widehat{\xi}$ and $\Phi=\widehat{\Phi}.$ That is $\mathcal{G}\circ\mathcal{F}=id_{\mathcal{S}(\mathcal{I})}.$

(2) To prove $\mathcal{F}\circ\mathcal{G}=id_{\mathcal{E}(\mathcal{I})},$ assume we have the following maps
$$\xymatrix@C=1cm{\mathcal{E}(A)\ar[d]^{\Psi=(\theta_0,\theta_1, \zeta)}\ar[r]^{\mathcal{G}}&\mathcal{S}(A)\ar[d]^{\Phi=(\theta_0,\theta_1,\xi)}\ar[r]^{\mathcal{F}}&\mathcal{E}(A)\ar[d]^{\widehat{\Psi}=(\theta_0,\theta_1,\widehat{\zeta})}\\
\mathcal{E}(B)\ar[r]^{\mathcal{G}}&\mathcal{S}(B)\ar[r]^{\mathcal{F}}&\mathcal{E}(B)},$$
where $\xi=\{\xi^p\}_{p\in\Sigma(A)}$, $\mathcal{G}(\Psi)=\Phi$ and $\mathcal{F}(\Phi)=\widehat{\Psi}.$ We need to show that $\widehat{\Psi}=\Psi.$

Let $\tau\in \text{T}B$, $\zeta(\tau)=\phi$ and $\widehat{\zeta}(\tau)=\widehat{\phi}$. Then
$$\widehat{\phi}|_{\overline{eAe}}=\widehat{\zeta}(\tau)|_{\overline{eAe}}=\xi^e(\tau_f)=\zeta(\tau)|_{\overline{eAe}}=\phi|_{\overline{eAe}}$$
for all pairs $e\in \mathcal{P}(A), ~f\in \mathcal{P}(B)$ with $\theta_0(e)=f$ and $f\in B^{\tau}$.
Therefore, $\widehat{\zeta}(x)=\zeta(x)$ for all $x$ in the ideal $I$ which is generated by projections in the set of $\theta_0^{-1}(B^{\tau}).$
By Lemma \ref{lemma1}, $\phi(x)=+\infty$ for all $x\in A\setminus I$.
By our construction of $\widehat{\phi}$, we know $\widehat{\phi}(x)=+\infty$ for all $x\in A\setminus I$.
Therefore, $\widehat{\phi}=\phi$, which completes the proof.

\end{proof}

\begin{remark}
Theorem \ref{mainthm2} follows from the above theorem.
\end{remark}




\begin{example}\label{countex}[Counter-Example]
In \cite{SR}, Shaloub Razak classified a class of $\text{C}^*$-algebras which are inductive limits of certain specified building blocks by using
their Elliott invariants. 
Those $\text{C}^*$-algebras Razak classified are simple, stably  projectionless and with trivial $\text{K}$-theory.  
Their Elliott invariants cannot be derived from their Stevens invariants since their Stevens invariants are all trivial. 
\end{example}






\begin{thebibliography}{10}
\bibitem{BH} B. Blackadar and D. Handelman, \emph{Dimension functions and traces on $C^*$-algebras}, J. Funct. Anal. 45(1982), 297-340.
\bibitem{Da} M. Dadarlat, \emph{Reduction to dimension three of local spectra of real rank zero $\textrm{C}^*$-algebras}, J. Reine Angew. Math., 460 (1995), 189-212.
\bibitem{DG} M. Dadarlat and G. Gong, \emph{A classification result for approximately homogeneous $\textrm{C}^*$-algebras of real rank zero}, GAFA, Geom.
Funct. Anal., 7 (1997), 646-711.
\bibitem{DL} M. Dadarlat and T. Loring, \emph{Classifying $\textrm{C}^*$-algebras via ordered, mod-p K-theory}, Math. Ann., 305 (1996), 601-616.
\bibitem{Ei} S. Eilers, \emph{A complete invariant for AD algebras
with bounded
dimension drop in $K_1$}, J. Funct. Anal., 139 (1996), 325-348.
\bibitem{E1} G. A. Elliott, \emph{On the classification of inductive limits of sequences of semi-simple
finite-dimensional algebras}, J. Algebra 38 (1976), 29–44.
\bibitem{E} G. A. Elliott, \emph{A classification of certain simple $C^*$-algebras}, Quantum and Non-
Commutative Analysis (editors, H. Araki et al.), Kluwer, Dordrecht, (1993), pp. 373-385.
\bibitem{Ell 1} G. A. Elliott, \emph{A classification of certain
simple
$\textrm{C}^*$-algebras}, II, J. Ramanujan Math. Soc., 12 (1) (1997), 97-134.
\bibitem{EG} G. A. Elliott and G.Gong, \emph{On the classification of $C^*$-algebras of real rank zero. II.}
Ann. of Math. 144 (1996), 497-610.
\bibitem{EG 2} G. A Elliott and G. Gong, \emph{On inductive limits
of matrix algebras over the two-torus}, Amer. J. of Math., 118
(1996),
263-290.
\bibitem{EGL} G. A. Elliott, G. Gong, and L.Li, \emph{On the classification of simple inductive limit $C^*$-
algebra II. The isomorphism theorem.} Invent. Math. 168 (2007), 249-320.
\bibitem{EGLP 1} G. A. Elliott, G. Gong, H. Lin and C. Pasnicu,
\emph{Abelian $\textrm{C}^*$-subalgebras of $\textrm{C}^*$-algebras
of real rank zero and inductive limit
$\textrm{C}^*$-algebras}, Duke Math. J., 85 (1996), 511-554.
\bibitem{EGLP 2} G. A. Elliott, G. Gong, H. Lin and C. Pasnicu,
\emph{Homomorphisms, homotopies and approximations by circle
algebras}, C.
R. Math. Rep. Acad. Sci. Canada, XVI (1994), 45-50.
\bibitem{ERS} G. A. Elliott, L. Robert and L. Santiago, \emph{The cone of lower semicontinuous traces on a $C^*$-algebra}, American Journal of Mathematics, Vol 133, No. 4, Aug. 2011, 969-1005.
\bibitem{FL} B. Fuchssteiner and W. Lusky, \emph{Convex Cones}, Mathematics Studies, Vol 56, Elsevier, 2011. 
\bibitem{Go 1} G. Gong, \emph{On inductive limits of matrix algebras
over higher
dimensional spaces}, Part II, Math. Scand., 80 (1997), 56-100.
\bibitem{Go 2} G. Gong, \emph{On the classification of
$\textrm{C}^*$-algebras of real rank
zero and unsuspended E-equivalence types}, J. Funct. Anal., 152(2):281-329, 1998.
\bibitem{Go 3} G. Gong, \emph{On the classification of simple
inductive limits $\textrm{C}^*$-algebras, I: The Reduction Theorem},
Documenta Math. 7 (2002),
255-461.
\bibitem{GJL} G. Gong, Chunlan Jiang, L. Li, \emph{A classification of inductive limit$\text{C}^*$-algebras with the ideal property},
https://arxiv.org/pdf/1607.07581v1.pdf.
\bibitem{GJLP} G. Gong, Chunlan Jiang, L. Li, C. Pasnicu, \emph{A
$\mathbb{T}$ structure for $\textrm{AH}$ algebras with the ideal
property and torsion free K-theory},
J. Funct. Anal.(2009), doi:10.1016/j.jfa.2009.11.016.
\bibitem {JJ} K. Ji and C. L. Jiang, \emph{A complete classification
of $AI$ algebras with the ideal property}, Canad. J. Math. 63(2011), 381-412.
\bibitem{JS} X. Jiang and H. Su, \emph{A classification of Simple
Limits of Splitting Interval Algebras}, J. Funct. Anal., no.
151. 50-76 (1997).
\bibitem{Li 1} L. Li, \emph{On the classification of simple
$\textrm{C}^*$-algebras : Inductive limits of matrix algebras over
tress}, Mem. Amer. Math. Soc., no.
605, vol. 127, 1997.
\bibitem{JW} C. Jiang and K. Wang, \emph{A complete classification of limits of splitting
interval algebras with the ideal property}, J. Ramanujan Math. Soc. 27, No.3 (2012) 305-354.
\bibitem{KR} E. Kirchberg and M. Rordam, \emph{Infinite non-simple C*-algebras: absorbing the Cuntz algebra $\mathcal{O}_{\infty}$}, Adv. Math. 167(2002), 195-264.
\bibitem{Lin book} Huaxin Lin, \emph{An Introduction to the classification of amenable $C^*$-algebras}, World Scientific.
\bibitem{SR} Shaloub Razak, \emph{On the classification of simple stably projectionless $\text{C}^*$-algebras}, Canad. J. Math 54, (2002), 
138-224. 
\bibitem{MR1} M. R{\o}rdam, \emph{On the structure of simple $C^*$-algebras tensored with a {\rm{UHF}}-algebra. II.}, J. Funct. Anal. 107(1992), 255-269.
\bibitem{MR2} M. R{\o}rdam, \emph{The stable and the real rank of $\mathcal{Z}$-absorbing $C^*$-algebras}, Int. J. Math. 15(2004), 1065-1084.
\bibitem{Pa1} C. Pasnicu, \emph{Shape equivalence, non-stable
K-theory and $\textrm{AH}$
algebras}, Pacific J. Math. 192 (2000), 159-182.
\bibitem{Pa2} C. Pasnicu, \emph{The ideal property, the projection property, continuous fields and crossed products},
J. Math. Anal. Appl. 323(2006), 1213-1224.
\bibitem{Pa3}  C. Pasnicu, \emph{The ideal property and tensor products of $C^*$-algebras, Rev. Roumaine Math},
Pures Appl. 49(2004), 153-162.
\bibitem{PP1} C. Pasnicu, N. C. Phillips, \emph{Permanence properties for crossed products and fixed point algebras of finite groups}, 
Trans. Amer. Math. Soc. 366(2014), 4625-4648.
\bibitem{PP2} C. Pasnicu, N. C. Phillips, \emph{The weak ideal property and topological dimension zero}, https://arxiv.org/abs/1601.00039.
\bibitem{PR} C. Pasnicu, M. R{\o}rdam, \emph{Tensor products of $C^*$-algebras with the ideal property}, J. Funct. Anal. 177(2000), 130-137.
\bibitem{Pedersen} G. K. Pedersen, \emph{$C^*$-algebras and their automorphism groups}, Vol. 14. Academic Pr, 1979.
\bibitem{Ph} R. R. Phelps,  \emph{Lectures on Choquet's theorem.} Springer Science \& Business Media, 2001.
\bibitem{PT} F. Perera and A. S. Toms, \emph{Recasting the Elliott Conjecture}, Math. Ann. 338, (2007), pp. 669-702.
\bibitem {Ste} K. H. Stevens, \emph{The classification of certain
non-simple approximate interval algebras}, Fields Institute
Communications 20 (1998), 105-148.
\end{thebibliography}
\end{document}